\documentclass[12pt]{article}
\usepackage{amssymb,latexsym,amsmath}
\usepackage{graphicx,times}
\usepackage[margin=1.0in]{geometry}
\usepackage{multirow}
\usepackage{enumerate}
\usepackage{amsthm}
\newtheorem{The}{Theorem}[section]
\newtheorem{Def}{Definition}[section]

\newtheorem{cor}{Corollary}

\numberwithin{equation}{section}
\begin{document}
\begin{center}
{\LARGE {\bf A new class of special functions arising from the solution of differential equations involving multiple proportional delays  }}
\vskip 0.5cm
{\Large Jayvant Patade$^a$\footnote{Corresponding author}, Sachin Bhalekar$^b$}\\
\textit{$^a$Ashokrao Mane group of InstitutionKolhapur - 416112, India.\\$^b$Department of Mathematics, Shivaji University, Kolhapur - 416004, India.\\ Email: $^a$dr.jayvantpatade@gmail.com,  $^b$sbb\_maths@unishivaji.ac.in, }\\
\end{center}

\begin{abstract}
  Proportional delay is a particular case of time dependent delay. In this article, we consider differential equations involving multiple delays. The series solution of this equation leads to a class of special functions. This class of special functions is independent from all the existing special functions obtained as a solution of differential equations. We analyze the basic properties of this class and discuss various identities and relations.
  
\end{abstract}
Keywords: Special functions, Successive approximation, proportional delay.

\section{Introduction}

Special functions are plying a vital role in the solutions of differential equations. Exponential, sine, cosine, hypergeometric and Mittag-Leffler are important classes of special functions arising as solutions of various classical and fractional differential equations. Though there is a huge literature devoted to the special functions arising from ordinary differential equations (ODEs), there is a lack of corresponding literature in delay differential equations (DDE).
In \cite{bp,pb1} we discussed the solutions of a class of DDE viz. proportional delay differential equations and provided the solutions in terms of new special functions. In this article, we generalize the DDEs considered in \cite{pb1} to involve multiple delays. Since the differential equations without delay are inequivalent to the differential equations involving delay, the special functions proposed in this article are independent from the existing special functions.

\section{Preliminaries}\label{Pre}
In this section, we discuss some basic definitions and results \cite{Magnus,Kilbas}.

\begin{Def}
The upper and lower incomplete gamma functions are defined as
\begin{equation}
\Gamma(n,x) = \int_{x}^\infty  t^{n-1} e^{-t}dt\quad \textrm{and}
\end{equation}
\begin{equation}
\gamma(n,x) = \int_{0}^x  t^{n-1} e^{-t}dt \quad \textrm{respectively}.
\end{equation}
\end{Def}

\begin{Def}
Kummer's confluent hyper-geometric functions $_{1}F_{1}(a;c;x)$ and  $U(a;c;x)$ are defined as below
\begin{equation}
_{1}F_{1}(a;c;x) = \sum_{n=0}^\infty \frac{(a)_n}{(c)_n} \frac{x^n}{n!},\quad c\ne 0,-1,-2,\cdots  \textrm{and} 
\end{equation}
\begin{equation}
U(a;c;x) = \frac{\pi}{\sin(\pi c)}\left(\frac{_{1}F_{1}(a;c;x)}{\Gamma(c)\Gamma(1+a-c)}-x^{1-c}\frac{_{1}F_{1}(1+a-c;2-c;x)}{\Gamma(a)\Gamma(2-c)}\right),
\end{equation}
\begin{equation}
-\pi<arg (x) \leq \pi.\nonumber
\end{equation}
\end{Def}

\begin{Def}\label{6.2.1}
The generalized Laguerre polynomials are defined as
\begin{eqnarray}
L_{n}^{(\alpha)}(x) &=& \sum_{m=0}^n (-1)^m \dbinom{n+\alpha}{n-m}\frac{x^m}{m!}\\
&=& \dbinom{n+\alpha}{n} {}_{1}F_{1}(-n;\alpha+1;x).
\end{eqnarray}
\end{Def}

\begin{Def}
	A real function $f(x)$, $x>0$, is said to be in space $C_\alpha$, $\alpha\in\mathbb{R}$, if there exists a real number $p (>\alpha)$, such that $f(x)=x^p f_1(x) $ where $f_1(x)\in C[0,\infty)$.
\end{Def}

\begin{Def}
	A real function $f(x)$, $x > 0$, is said to be in space $C^m_\alpha$, $m\in\mathbb{N}\cup \{0\}$, if $ f^{(m)} \in C_\alpha$.
\end{Def}
\begin{Def}
	Let $f\in C_\alpha $ and $\alpha \geq -1$, then the (left-sided) Riemann-Liouville integral of order $\mu, \mu> 0 $ is given by
	\begin{equation}
	I^\mu f(t)=\frac{1}{\Gamma(\mu)} \int_{0}^t (t-\tau)^{\mu-1}f(\tau)d\tau,\quad t>0.
	\end{equation}
\end{Def}

\begin{Def}
	The (left sided) Caputo fractional derivative of $f, f \in C_{-1}^m, m\in\mathbb{N}\cup\{0\}$, is defined as:
	\begin{eqnarray}
	D^\mu f(t)&=&\frac{d^m}{ dt^m} f(t),\quad \mu = m \nonumber\\
	&=& I^{m-\mu}\frac{d^m}{ dt^m} f(t),\quad {m-1} <\mu <m,\quad m\in \mathbb{N}.
	\end{eqnarray}
\end{Def}
Note that for $0\le m-1 < \alpha \le m$ and $\beta>-1$
\begin{eqnarray}
I^\alpha (x-b)^\beta &=&\frac{\Gamma{(\beta+1)}}{ \Gamma{(\beta+\alpha+1)}} (x-b)^{\beta+\alpha},\nonumber\\
\left(I^\alpha D^\alpha f\right)(t)&=& f(t)-\sum_{k=0}^{m-1} f^{(k)}(0)\frac{t^k}{k!}.
\end{eqnarray}

\begin{Def}
	Mittag-Leffler function of order $\alpha>0$ is defined by the series 
	\begin{equation}
	E_\alpha (x)=\sum_{n=0}^\infty\frac{x^n}{\Gamma{(\alpha n+1)}}
	\end{equation}
\end{Def}

\subsection{Daftardar-Gejji and Jafari Method}
Daftardar-Gejji and Jafari Method (DJM) \cite{GEJJI1} is one of the popular methods applied to solve nonlinear equations  of form 
\begin{equation}
u= f + L(u) + N(u),\label{1.1.1}
\end{equation}
where $L$ and $N$ are linear and nonlinear operators respectively and $f$ is known function.\\
In this case, the DJM provides the solution in the form of series 
\begin{equation}
u= \sum_{i=0}^\infty u_i=f +  \sum_{i=0}^\infty L(u_i) + \sum_{i=0}^\infty G_i\label{1.1.2}
\end{equation}
where $G_0=N(u_0)$ and $G_i = \left\{N\left(\sum_{j=0}^i u_j\right)- N\left(\sum_{j=0}^{i-1} u_j\right)\right\}$, $i\geq 1$.\\
From Eq.(\ref{1.1.2}), the DJM series terms are generated as bellow:
\begin{equation}
u_0 =f,\quad u_{m+1}= L(u_m) + G_m,\quad m=0,1, 2, \cdots.\label{1.1.6}
\end{equation}

\section{Existence, uniqueness and convergence:Nonlinear Case}\label{exist}
 First, we consider the nonlinear equation
\begin{equation}
y'(x) = f\left(x, y(q_0x), y(q_1x),\cdots,y(q_nx)\right),  q_0=1\, \textrm{and} \, 0<q_i<1,i=1,2,\cdots,n. \label{1}
\end{equation}
The Eq. (\ref{1}) is a particular case of time dependent delay differential equation (DDE)
\begin{equation}
y'(x) = f\left(x, y\left(x-\tau_0(x)\right), y\left(x-\tau_1(x)\right),y\left(x-\tau_2(x)\right)\cdots y\left(x-\tau_n(x) \right)\right)\nonumber
\end{equation}
with\quad$\tau_i(x)=(1-q_i)x,\quad i=1,2,3,\cdots,n$. The DJM series solution of Eq. (\ref{1}) is of the form
 \begin{equation}
 y= \sum_{i=0}^\infty y_i. \label{2}
 \end{equation}

We present convergence result of this series solution motivated from \cite{pb3}.
\begin{The}
Let f be a continuous function defined on a $(n+2)$ dimensional rectangle\\ $R=\{ (x,y_0,y_1,\cdots,y_n)| 0\leq x\leq b, -\delta_i\leq y_i\leq \delta_i, i=0,1,2,\cdots,n \} $ and $\mid f\mid \leq M$ on $R$. Suppose that f satisfies Lipschitz type condition \\
 \begin{equation}
\mid f\left(x,y_0,y_1,\cdots,y_n\right)- f\left(x,u_0,u_1,\cdots,u_n\right)\mid \leq \sum_{i=0}^nL_i \mid y_i- u_i \mid.\label{3}
\end{equation}
 Then the DJM series solution  (\ref{2}) of DDE  (\ref{1})
converges uniformly  in the interval  [0,b].
\end{The}

\begin{proof}
Without loss of gerenality, assume that $y(0) =1$.	
The equivalent integral equation of  (\ref{1}) is
\begin{equation}
y(x)=1 + \int_{0}^{x} f\left(x, y(q_0t), y(q_1t),y(q_2t),\cdots,y(q_nt)\right)dt.\nonumber
\end{equation}
Using DJM, we get
\begin{eqnarray*}
y_0(x) &=& 1,\\
y_1(x) &=& \int_{0}^{x} f\left(t, y_0(q_0t), y_0(q_1t),y_0(q_2t),\cdots,y_0(q_nt)\right)dt.\\
\Rightarrow\mid y_1(x)\mid &\leq& Mx.
\end{eqnarray*}
$\textrm{Since}\; q_0=1, 0<q_i<1,i=1,2,3,\cdots,n., \; q_ib\leq b.$
\begin{eqnarray*}
\Rightarrow \mid y_1(q_ix)\mid &\leq& M q_i x, \quad  \forall x\in [0,b].
\end{eqnarray*}
\begin{eqnarray*}
y_2(x) &=& \int_{0}^{x} \left(f\left(t, y_1(q_0t)+y_0(q_0t), y_1(q_1t)+y_0(q_1t),y_1(q_2t)+y_0(q_2t),\cdots,y_1(q_nt)+y_0(q_nt)\right)\right.\\
&& \left. - f\left(t, y_0(q_0t), y_0(q_1t),y_0(q_2t),\cdots,y_0(q_nt)\right)\right)dt.\\
\Rightarrow\mid y_2(x)\mid &\leq & \int_{0}^{x} \left(\sum_{i=0}^nL_i \mid y_1(q_it)\mid\right)dt\\
&\leq& M\left(\sum_{i=0}^nL_iq_i\right)\frac{x^2}{2!}\\
&\leq& M\left(\sum_{i=0}^nL_i\right)\frac{x^2}{2!}.\\
\Rightarrow \mid y_2(q_ix)\mid &\leq& Mq_i^2\left(\sum_{i=0}^nL_iq_i\right)\frac{x^2}{2!},\quad x\in [0,b]\\
&\leq &   M\left(\sum_{i=0}^nL_i\right)\frac{x^2}{2!}.
\end{eqnarray*}

Using induction, we get
\begin{eqnarray*}
\mid y_m(x)\mid &\leq& M \prod_{j=1}^{m-1}\left(\sum_{i=0}^nL_iq^j_i\right)\frac{x^m}{m!}\\
 &\leq&   M\left(\sum_{i=0}^nL_i\right)^{m-1}\frac{x^m}{m!}, \quad m=1,2,3\cdots.
\end{eqnarray*}
Taking summation over $m$, we get

\begin{eqnarray*}
\left\vert \sum_{m=0}^\infty y_m \right\vert & \leq & \frac{M}{\sum_{i=0}^nL_i} \left(e^{\sum_{i=0}^nL_ix}-1\right)+1\\
& \leq & \frac{M e^{\sum_{i=0}^nL_ib}}{\sum_{i=0}^nL_i} +1 , \quad x\in [0,b].
\end{eqnarray*}
Thus, by \cite{Rudin}, we can conclude that the series solution of (\ref{1}) converges uniformly  in the interval  [0,b]. Hence existence of (\ref{1}) is proved. The uniqueness is of solution is obvious by condition (\ref{3}).     
\end{proof}

\section{Stability analysis}\label{stab}
The following definitions and theorems are generalization of corresponding definition and theeorems given in \cite{Deng}.

\begin{Def}
Consider the autonomus time-dependent delay differential equation (DDE),
\begin{equation}
y'(x) = g(y\left(x\right), y\left(x-\tau_1(x)\right),y\left(x-\tau_2(x)\right)\cdots y\left(x-\tau_n(x) \right)),\label{6.2.3}
\end{equation}
where $g:\mathbb{R}^{n+1} \rightarrow \mathbb{R}$.
The flow $\phi_x(x_0)$ is the solution $y(x)$ of ({\ref{6.2.3}}) with initial condition\\ $y(x) = x_0,\, x\leq 0$. The point $y^*$ is called equilibrium solution of  ({\ref{6.2.3}}) if $g(y^*,y^*,\cdots, y^* )=0$.\\
\textbf{(a)} If, for any  $\epsilon > 0$, there exist  $\delta > 0$ such that $ |x_0-y^*|< \delta \Rightarrow |\phi_x(x_0)-y^*| < \epsilon,$ then the system ({\ref{6.2.3}}) is stable (in the Lyapunov sense) at the equilibrium  $y^*$.\\
\textbf{(b)} If the system ({\ref{6.2.3}}) is stable at   $y^*$ and moreover,$\lim\limits_{x\rightarrow\infty}|\phi_x(x_0)-y^*|=0$ then the system  ({\ref{6.2.3}}) is said to be asymptotically stable at $y^*$.\\ 
\textbf{(c)} If the system ({\ref{6.2.3}}) is not stable then it is called unstable.
\end{Def}

\begin{The}
Assume that the  equilibrium solution $y^*$ of the equation
\begin{equation}
y' = g(y(x),y(x-\tau_1^*),y(x-\tau_2^*),\cdots,y(x-\tau_n^*) ),\quad \tau_1^*= \tau_1(x_0), \tau_2^*= \tau_2(x_0),\cdots,\tau_n^*= \tau_n(x_0) \nonumber 
\end{equation}
is stable at equilibrium  $y^*$ and
\begin{equation}
\|g(y(x),y(x-\tau_1(x)),\cdots,y(x-\tau_n(x))) - g(y(x),y(x-\tau_1(x_1)),\cdots,y(x-\tau_n(x_n)))\|<\sum_{i=0}^n\epsilon_i |x-x_i|,\nonumber
\end{equation}
for some $\epsilon_i>0$ and $x,x_i\in [x_0, x_0+c)$$(i=1,2,\cdots,n)$,
c is a positive constant, then there exists $\bar{x}> 0$ such that
the equilibrium solution $y^*$ of Eq. ({\ref{6.2.3}}) is stable at equilibrium  $y^*$ on finite time interval $[x_0,\bar{x})$.
\end{The}

\begin{cor}
If the real parts of all roots of $ \lambda -\sum_{i=0}^na_ie^{-\lambda\tau_i^*} $ are negative, where $ a_i =\partial_if$, $i=1,2,\cdots,n$ evaluated at equilibrium $y^*$, then  Eq. ({\ref{6.2.3}}) is stable at $y^*$ on finite time
interval  $[x_0,\bar{x})$.
\end{cor}

\section{Linear equation: Exact solution}\label{panto}

Consider the differential equation involving multiple proportional delays,

\begin{equation}
y'(x) = \sum_{i=0}^n a_i y(q_ix),\quad y(0)= 1,  \label{6.3.1}
\end{equation}
where $ q_0=1$, $0<q_i<1$ and $a_0,a_i \in \mathbb{R}$, $i=1,2,3,\cdots,n$.
The equation (\ref{6.3.1}) has applications in Science and Engineering \cite{pb2,Ockendon}. \\
Integrating (\ref{6.3.1}), we get
\begin{equation}
y(x)= 1 + \int_{0}^{x}\left(\sum_{i=0}^n a_i y(q_it)\right)dt\label{6.3.2}\nonumber
\end{equation}
Using successive approximation, we obtain
\begin{eqnarray}
y_0(x) &=& 1 ,\nonumber\\
y_1(x)&=& \int_{0}^{x}\left(\sum_{i=0}^n a_i y_0(q_it)\right)dt\nonumber\\
&=& \sum_{i=0}^n a_i\frac{x}{1!},\nonumber\\
y_2(x)&=& \int_{0}^{x}\left(\sum_{i=0}^n a_i y_1(q_it)\right)dt\nonumber\\
&=&\int_{0}^{x}\sum_{i=0}^n a_i\left(\sum_{i=0}^n a_iq_i t\right)dt\nonumber\\
&=& \left(\sum_{i=0}^n a_i\right)\left(\sum_{i=0}^n a_iq_i\right)\frac{x^2}{2!},\nonumber\\
y_3(x)&=&  \int_{0}^{x}\left(\sum_{i=0}^n a_i y_2(q_it)\right)dt\nonumber\\
&=& \left(\sum_{i=0}^n a_i\right)\left(\sum_{i=0}^n a_iq_i\right)\left(\sum_{i=0}^n a_iq^2_i\right)\frac{x^3}{3!},\nonumber\\
&\vdots& \nonumber
\end{eqnarray}

\begin{eqnarray}
y_m(x)&=& \frac{x^m}{m!}\prod_{j=0}^{m-1}\left(\sum_{i=0}^n a_iq^j_i\right), \quad m=1,2,3\cdots. \nonumber
\end{eqnarray}
$\therefore$ The exact solution of  (\ref{6.3.1}) is 
\begin{eqnarray}
y(x) &=& y_0(x)+y_1(x)+y_2(x)+\cdots\nonumber\\
&=&  1 + \sum_{m=1}^\infty \frac{x^m}{m!}\prod_{j=0}^{m-1}\left(\sum_{i=0}^n a_iq^j_i\right).\label{6.3.3}
\end{eqnarray}
If we define 
\begin{equation}
\prod_{j=0}^{m-1}\left(\sum_{i=0}^n a_iq^j_i\right) =1 \quad \textrm{for}\quad m=0,\nonumber
\end{equation}
then 
\begin{equation}
y(x)=\sum_{m=0}^\infty \frac{x^m}{m!}\prod_{j=0}^{m-1}\left(\sum_{i=0}^n a_iq^j_i\right)
\end{equation}
This solution of (\ref{6.3.1}) provides a novel special function 
\begin{equation}
 \mathcal{R}(\bar{a};\bar{q};x)= \sum_{m=0}^\infty \frac{x^m}{m!}(\bar{a};\bar{q})_m.\label{6.3.4}
\end{equation}
\textbf{Note:} We use a brief notations $\mathcal{R}(\bar{a};\bar{q};x)$ for the special function $\mathcal{R}(a_0,a_1,\cdots,a_n;q_0,q_1,\cdots,q_n,x)$ and $(\bar{a};\bar{q})_m$ for $\prod_{j=0}^{m-1}\left(\sum_{i=0}^n a_iq^j_i\right)$.

\section{Analysis}\label{analy}

\begin{The}
	The power series
\begin{equation}
\mathcal{R}(\bar{a};\bar{q};x)= \sum_{m=0}^\infty \frac{x^m}{m!}(\bar{a};\bar{q})_m,\label{6.5.0}
\end{equation}

 has infinite radius of convergence.
\end{The}

\begin{proof}
Suppose
\begin{equation*}
A_m =\frac{1}{m!}(\bar{a};\bar{q})_m,\quad m=0,1,2,\cdots.
\end{equation*}
If  $R$ is radius of convergence of (\ref{6.5.0}) then by using ratio test  \cite{Apostal}
\begin{eqnarray}
\frac{1}{R}=\lim_{m\to\infty}\left| \frac{A_{m+1}}{A_m}\right| &=&\lim_{m\to\infty}\left| \frac{ \frac{1}{{(m+1)!}}\prod_{j=0}^{m}\left(\sum_{i=0}^n a_iq^j_i\right)}{ \frac{1}{m!}\prod_{j=0}^{m-1}\left(\sum_{i=0}^n a_iq^j_i\right)}\right| \nonumber\\
&=& \lim_{m\to\infty} \left| \frac{\left(\sum_{i=0}^n a_iq^m\right)}{(m+1)}\right| \nonumber\\
\Rightarrow R&=& \infty.\nonumber
\end{eqnarray}
Thus the series has infinite radius of convergence.
\end{proof}  

\begin{cor}
The power series (\ref{6.3.4}) is  absolutely convergent for all $x$ and hence it is uniformly convergent on any compact interval on $\mathbb{R}$.
\end{cor}

\begin{The}\label{6.1.3}
For $r\in \mathbb{N} \cup\{0\}$, we have 
\begin{equation}
\frac{d^r}{dx^r}\mathcal{R}(\bar{a};\bar{q};x) =\mathcal{R}^{(r)}(\bar{a};\bar{q};x)= \sum_{m=r}^\infty \frac{x^{m-r}}{(m-r)!}(\bar{a};\bar{q})_m.\nonumber\\
\end{equation}
\end{The}

\begin{The}(Addition Theorem) 
\begin{equation}
\mathcal{R}(\bar{a};\bar{q};x+y)=\sum_{r=0}^\infty \frac{x^r}{r!}\mathcal{R}^{(r)}(\bar{a};\bar{q};y)\nonumber
\end{equation}
\end{The}

\begin{proof}
We have
\begin{eqnarray}
\mathcal{R}(\bar{a};\bar{q};x+y)&=&  \sum_{m=0}^\infty \frac{(x+y)^m}{m!}(\bar{a};\bar{q})_m\nonumber\\
&=&  \sum_{m=0}^\infty \sum_{r=0}^m \frac{x^r}{r!} \frac{y^{m-r}}{(m-r)!}(\bar{a};\bar{q})_m\nonumber\\
&=& \sum_{r=0}^\infty \sum_{m=r}^\infty \frac{x^r}{r!} \frac{y^{m-r}}{(m-r)!}(\bar{a};\bar{q})_m\nonumber\\
&=& \sum_{r=0}^\infty \frac{x^r}{r!} \sum_{m=r}^\infty  \frac{y^{m-r}}{(m-r)!}(\bar{a};\bar{q})_m\nonumber\\
\mathcal{R}(\bar{a};\bar{q};x+y) &=& \sum_{r=0}^\infty \frac{x^r}{r!}\mathcal{R}^{(r)}(\bar{a};\bar{q};y).\nonumber
\end{eqnarray}
\end{proof}

\textbf{ Note:}
If $x+iy=z\in\mathbb{C}$, then
\begin{equation}
\mathcal{R}(\bar{a};\bar{q};z)=\sum_{r=0}^\infty \frac{(-ix)^r}{r!}\mathcal{R}^{(r)}(\bar{a};\bar{q};iy)\nonumber
\end{equation}

\begin{The} 
If $z=re^{i\theta}\in\mathbb{C}$, then
\begin{equation}
\mathcal{R}(\bar{a};\bar{q};z)=\mathcal{R}_c(\bar{a};\bar{q};z)+i\mathcal{R}_s(\bar{a};\bar{q};z),\nonumber
\end{equation}
where, 
\begin{eqnarray}
\mathcal{R}_c(\bar{a};\bar{q};z)&=\textrm{Real part of\, }\mathcal{R}(\bar{a};\bar{q};z)&=\sum_{m=0}^\infty \frac{r^m\cos\theta}{m!}(\bar{a};\bar{q})_m\nonumber\\
\mathcal{R}_s(\bar{a};\bar{q};z)&=\textrm{Imaginary  part of\, }\mathcal{R}(\bar{a};\bar{q};z)&=\sum_{m=0}^\infty \frac{r^m\sin\theta}{m!}(\bar{a};\bar{q})_m,\nonumber
\end{eqnarray}
\end{The}

\begin{The}
For  $ q_0=1$, $0<q_i<1$ and $a_0,a_i \ge 0$, $i=0,1,2,3,\cdots,n$.  The function $\mathcal{R}(\bar{a};\bar{q};x)$ satisfies the following inequality
\begin{equation}
 e^{a_0x}\le \mathcal{R}(\bar{a};\bar{q};x)\le e^{\left(\sum_{i=0}^n a_i\right)x},\quad 0 \le x < \infty.\nonumber
\end{equation}
\end{The}

\begin{proof}
 Since $ q_0=1$, $0<q_i<1$ and $a_0,a_i \ge 0$, $i=0,1,2,3,\cdots,n$, we have
\begin{eqnarray}
\prod_{j=0}^{m-1}\left(\sum_{i=0}^n a_iq^j_i\right) &\le& \left(\sum_{i=0}^n a_i\right)^m \nonumber\\
\Rightarrow  \frac{x^m}{m!}\prod_{j=0}^{m-1}\left(\sum_{i=0}^n a_iq^j_i\right) &\le&  \frac{x^m \left(\sum_{i=0}^n a_i\right)^m}{m!}.\nonumber\\
\textrm{Taking summation over m, we get} \nonumber\\
 \mathcal{R}(\bar{a};\bar{q};x)\le e^{\left(\sum_{i=0}^n a_i\right)x},\quad 0 \le x < \infty.\label{6.5.1}
\end{eqnarray}
Similarly, we have
\begin{eqnarray}
a_0^m &\le& \prod_{j=0}^{m-1}\left(\sum_{i=0}^n a_iq^j_i\right)  \nonumber\\
\Rightarrow e^{a_0x} &\le& \mathcal{R}(\bar{a};\bar{q};x) ,\quad 0 \le x < \infty.\label{6.5.2}
\end{eqnarray}
From (\ref{6.5.1}) and (\ref{6.5.2}), we get
\begin{equation}
 e^{a_0x}\le \mathcal{R}(\bar{a};\bar{q};x)\le e^{\left(\sum_{i=0}^n a_i\right)x},\quad 0 \le x < \infty.\nonumber
\end{equation}
\end{proof}

\section{Generalization to fractional order DDE }\label{gen}
Consider the fractional delay differential equation involving multiple proportional delays,
\begin{equation}
D^\alpha_0y(x) =\sum_{i=0}^n a_i y(q_ix),\quad y(0)= 1, \label{7.1.1}
\end{equation}
where $ q_0=1$, $0<q_i<1$ and $a_0,a_i \in \mathbb{R}$, $i=1,2,3,\cdots,n$.

The exact solution of (\ref{7.1.1}) is

\begin{equation*}
\mathcal{R}_\alpha(\bar{a};\bar{q};x) = \sum_{m=0}^\infty \frac{x^{\alpha m}}{\Gamma{(\alpha m + 1)}}(\bar{a};\bar{q})_{\alpha,m}\label{7.1.3}
\end{equation*}
where $(\bar{a};\bar{q})_{\alpha,m}=\prod_{j=0}^{m-1}\left(\sum_{i=0}^n a_iq^{\alpha j}_i\right)$
\begin{The}
The power series  (\ref{7.1.3}) is convergent for all finite values of $x$.
\end{The}

\begin{The}

For  $ q_0=1$, $0<q_i<1$ and $a_0,a_i \ge0$, $i=0,1,2,3,\cdots,n$.  The function $\mathcal{R}_\alpha(\bar{a};\bar{q};x)$ satisfies the following inequality

	\begin{equation}
	E_\alpha{(a_0x^\alpha)}\le \mathcal{R}_\alpha(\bar{a};\bar{q};x)\le E_\alpha{\left(\left(\sum_{i=0}^n a_i\right)x^\alpha\right)},\quad 0 \le x < \infty.\nonumber
	\end{equation}
\end{The}

\section{Conclusions}\label{concl}
In this paper, we have obtained a new special function arising from  differential equation involving multiple proportional delays. The solution is obtained by applying the successive approximation method. The existence, uniqueness, stability and convergence results for the time dependent delay differential equations are presented in this paper. The new special function exhibit different properties and relations. The generalization to fractional order case is also presented.

\textbf{Acknowledgements:}\\
S. Bhalekar acknowledges the Science and Engineering Research Board (SERB), New Delhi, India for the Research Grant (Ref. MTR/2017/000068) 
under Mathematical Research Impact Centric Support (MATRICS) Scheme.

\end{document}